\documentclass[12pt]{amsart}

\usepackage{amsfonts, amssymb, amsmath, latexsym,  mathtools, amscd,mathrsfs}
\usepackage{hyperref,graphicx}
\usepackage{xcolor}
\usepackage{url}
\usepackage{listings}
\usepackage{wrapfig}
\usepackage{subcaption,cleveref}
\usepackage{cite}
\usepackage{tikz}

\newtheorem{Theorem}{Theorem}[section]

\newtheorem{Lemma}[Theorem]{Lemma}
\newtheorem{Corollary}[Theorem]{Corollary}

\theoremstyle{remark}
\newtheorem{Remark}[Theorem]{Remark}

\newcommand{\CC}{{\mathbb C}}

\newcommand{\PP}{{\mathbb P}}

\newcommand{\ZZ}{{\mathbb Z}}

\newcommand{\calU}{{\mathcal U}}
\newcommand{\calX}{{\mathcal X}}
\newcommand{\calR}{{\mathcal R}}

\newcommand{\calN}{{\mathcal N}}
\newcommand{\conv}{{conv}}

\title[Reducibility for Finite Jacobi Pencils]{Reducibility and Diagonal Equality Strata for Finite Jacobi Pencils}

\author{Matthew Faust}
\address{Matthew Faust, Department of Mathematics, Michigan State University, East Lansing, MI 48824, USA} \email{mfaust@msu.edu}
\urladdr{https://mattfaust.github.io/}



\begin{document}

\begin{abstract}
We study reducibility of the characteristic polynomials of a class of finite Jacobi pencils.  We classify the diagonal equality patterns for which generic connected couplings yield reducibility and use this classification to prove that for $n\ge4$, the connected reducible locus has codimension at least two in the connected parameter space. For $n=2m+1$, we also prove that for every fixed connected coupling vector, the locus of diagonal potentials that admit a constant branch has codimension $m$.
\end{abstract}
\maketitle 
    

\section{Introduction}
We consider reducibility of the characteristic polynomial of Jacobi matrices of the form
\[J_n(w)=
\begin{pmatrix}
 a_1 & wb_1 & 0 & \cdots & 0\\
 wb_1 & a_2 & wb_2 & \ddots & \vdots\\
 0 & wb_2 & a_3 & \ddots & 0\\
 \vdots & \ddots & \ddots & \ddots & wb_{n-1}\\
 0&\cdots&0&wb_{n-1}&a_n
\end{pmatrix}.\]
We use the sign convention
\[  \chi_n(w,\lambda)=\det(J_n(w)+\lambda I). \]
Since $w$ only appears in the off-diagonal entries, $\chi_n$ is even in $w$.
We therefore write
\[ \chi_n(w,\lambda)=P_n(t,\lambda),\qquad t=w^2,\]
where
\[  P_0=1,\qquad P_1=\lambda+a_1,\qquad P_k=(\lambda+a_k)P_{k-1}-t c_{k-1}P_{k-2},\]
with $c_i=b_i^2$. 

A recent preprint of Shapiro~\cite{shapiro2026reducibilityspectralcurvesfinite} studies when $\chi_n(w,\lambda)$ is reducible.
Shapiro first proves that, for fixed pairwise-distinct diagonal entries $a_1,\ldots,a_n$, the polynomial $\chi_n$ is irreducible for all $b=(b_1,\ldots,b_{n-1})$ outside a proper algebraic subset~\cite[Theorem 2.1]{shapiro2026reducibilityspectralcurvesfinite}.
The same paper then proposes that, for $n\ge4$, the only codimension-one components of the reducible locus are the cut hyperplanes $b_i=0$ (see 
~\cite[Conjecture 4.3]{shapiro2026reducibilityspectralcurvesfinite}).

In this paper, we prove \cite[Conjecture 4.3]{shapiro2026reducibilityspectralcurvesfinite} (Theorem~\ref{thm:no-divisors-full}). In addition to this, for any fixed connected coupling, we give the dimension of the class of odd length potentials that admit constant branches (Theorem~\ref{thm:constant-branches-fixed-couplings}). We also characterize generic reducibility for arbitrary fixed diagonal data: in odd length, the odd-indexed diagonal entries must all coincide, while in even length, each parity class must be constant (Theorem~\ref{thm:sec8main} and Corollary~\ref{cor:inheritance}). We note that the methods used are close in spirit to those developed for studying the reducibility of dispersion polynomials~\cite{Liu2022Fermi, FaustLopezGarcia2025, FillmanLiuMatos2022, faust2026genericirreducibilityblochvarieties}. Indeed, one can view $\chi_n(w,\lambda)$ as the dispersion polynomial of the discrete operator $J_n(w) = A + w\Delta_b$ on an $n$ vertex path graph, where $A$ is a diagonal matrix with entries $a_i$ (the potential) and $\Delta_b$ is a weighted graph adjacency operator with couplings $b_i$.

The paper is organized as follows.
In Section~\ref{sec:basic}, we collect elementary reductions.
In Section~\ref{sec:newton-polytope}, we use Newton polytopes to constrain possible factors.
In Section~\ref{sec:constant-branches}, we classify constant branches on the connected pairwise-distinct stratum and give the dimension count for constant branches.
In Section~\ref{sec:pairwise-distinct-strata}, we derive a first-order obstruction for each Hensel subset and prove the codimension bound for fixed pairwise-distinct diagonal data.
In Section~\ref{sec:generic-fixed-diagonal}, we classify generic reducibility for arbitrary fixed diagonal data, both for $P_n$ and for $\chi_n$.
Finally, in Section~\ref{sec:dimension-bound}, we combine these results to prove the codimension bound on the full connected parameter space.

\begin{Remark}
    This manuscript was prepared in response to the first version of Shapiro’s preprint~\cite{shapiro2026reducibilityspectralcurvesfinite}. 
    A second version appeared~\cite{shapiro2026reducibilityspectralcurvesfinite2} while this manuscript was in preparation and contains some overlapping preliminary results.
    To our knowledge, the main results of this manuscript (Theorem~\ref{thm:constant-branches-fixed-couplings},  Theorem~\ref{thm:sec8main},  Corollary~\ref{cor:inheritance}, and Theorem~\ref{thm:no-divisors-full}) are new.
    We also note that these results do not require the diagonal entries of $J_n$ to be pairwise distinct, which is the primary setting of ~\cite{shapiro2026reducibilityspectralcurvesfinite2}. 
    For completeness, Sections~\ref{sec:basic}, \ref{sec:newton-polytope}, and \ref{sec:constant-branches} retain some results that overlap with the revised preprint.
\end{Remark}

\section{Basic Reductions}\label{sec:basic}

We begin by collecting notation and making some elementary reductions.
We call $\chi_n(w,\lambda)$ the spectral polynomial, $P_n(t,\lambda)$ the reduced continuant, and write
\[ A_i=\lambda+a_i.\]
For a contiguous block $[r,s]$, we write $P_{r,s}(t,\lambda)$ for the corresponding reduced continuant. 
That is, $P_{r,s}(t,\lambda)=\det(J_{r,s}(\sqrt{t})+\lambda I)$, where
\[J_{r,s}(w)=
\begin{pmatrix}
 a_r & wb_{r} & 0 & \cdots & 0\\
 wb_r & a_{r+1} & wb_{r+1} & \ddots & \vdots\\
 0 & wb_{r+1} & a_{r+2} & \ddots & 0\\
 \vdots & \ddots & \ddots & \ddots & wb_{s-1}\\
 0&\cdots&0&wb_{s-1}&a_s
\end{pmatrix}.\]

Recall that $c_i = b_i^2$, we say the chain is connected if
\[  c_1\cdots c_{n-1}\neq0.\]
If $c_i=0$, the chain cuts and the characteristic polynomial factors as the product of the two subchain characteristic polynomials.
We write
\[  T_c=(\CC^*)^{n-1}\]
for the connected coupling torus with coordinates $c_1,\ldots,c_{n-1}$.
Unless explicitly stated otherwise, we work over $T_c$.

\begin{Lemma}\label{lem:elementary-invariances}
The affine change $\lambda\mapsto\lambda+s$ preserves reducibility in both $\CC[w,\lambda]$ and $\CC[t,\lambda]$.
In particular, shifting all diagonal entries by the same scalar does not change reducibility.

If all diagonal entries are equal, say $a_1=\cdots=a_n=a$, then
\[ \chi_n(w,\lambda)=\det((\lambda+a)I+wB)  =\prod_{\mu\in\text{Spec}(B)}(\lambda+a+\mu w)\]
over $\CC$.
\end{Lemma}

\begin{proof}
The first statement is applying a polynomial-ring automorphism.
The second is the ordinary eigenvalue factorization of the scalar diagonal pencil.
\end{proof}

\begin{Lemma}\label{lem:matching-support}
For a block $[r,s]$,
\[ P_{r,s}(t,\lambda)  =  \sum_M  (-t)^{|M|}  \left(\prod_{e\in M}c_e\right)  \prod_{j\in[r,s]\setminus V(M)}A_j,\]
where the sum is over matchings of the path $[r,s]$.
Consequently, if $t^k\lambda^\ell$ appears in $P_n$, then
\[  2k+\ell\le n. \]
Equivalently, if $w^r\lambda^\ell$ appears in $\chi_n$, then $r$ is even and
\[ r+\ell\le n.\]
\end{Lemma}

\begin{proof}
The matching formula follows directly from the Leibniz expansion of the determinant.
The support inequalities are immediate, since a matching of size $k$ removes $2k$ vertices from the product of the $A_i$.
\end{proof}

\begin{Lemma}\label{lem:universal-support}
Let $F(t,\lambda)$ be a monic factor of $P_n(t,\lambda)$ of $\lambda$-degree $d$.
Write
\[ F(t,\lambda) =  \lambda^d+f_1(t)\lambda^{d-1}+\cdots+f_d(t).\]
Then
\[ \deg_t f_j\le \left\lfloor\frac j2\right\rfloor.\]
Equivalently, if
\[ F(t,\lambda)=F_0(\lambda)+tF_1(\lambda)+t^2F_2(\lambda)+\cdots,\]
then
\[ \deg_\lambda F_k\le d-2k.\]
In particular, monic linear factors are independent of $t$, and monic quadratic factors on the pairwise-distinct stratum have the form
\[ A_iA_j-\eta t.\]
\end{Lemma}

\begin{proof}
Let $P_n=FG$, where $G$ is monic of $\lambda$-degree $n-d$.
The Newton polytope of a product is the Minkowski sum of the Newton polytopes of the factors, and $P_n$ is supported in the halfspace $2k+\ell\le n$ by Lemma~\ref{lem:matching-support}.
Since $G$ is monic, $(0,n-d)\in N(G)$.
Thus, for every $(k,\ell)\in N(F)$, the point $(k,\ell+n-d)$ lies in $N(P_n)$, and hence
\[  2k+\ell+n-d\le n.\]
Therefore, $2k+\ell\le d$, which is the claimed support bound for $F$.
The final assertions are the cases $d=1$ and $d=2$, together with the specialization
\[  P_n(0,\lambda)=\prod_i A_i.\]
\end{proof}

\begin{Lemma}\label{lem:hensel-subsets}
Assume the $a_i$ are pairwise distinct.
If
\[ \chi_n(w,\lambda)=F(w,\lambda)G(w,\lambda)\]
with $F,G$ monic in $\lambda$, then $F$ and $G$ are even in $w$.
Hence, every factorization of $\chi_n(w,\lambda)$ descends to a factorization of $P_n(t,\lambda)$, and conversely.

Moreover, after setting $w=0$, a monic factor determines a unique subset
\[ S\subset\{1,\ldots,n\}\]
by
\[ F(0,\lambda)=\prod_{i\in S}A_i.\]
\end{Lemma}

\begin{proof}
At $w=0$,
\[ \chi_n(0,\lambda)=\prod_{i=1}^n A_i\]
is squarefree.
Thus, the specialization of a monic factor is the product of a unique subset of the $A_i$.
The involution $w\mapsto -w$ preserves $\chi_n$, and Hensel uniqueness at $w=0$ fixes this subset.
Therefore
\[ F(-w,\lambda)=F(w,\lambda),\]
and similarly for $G$.
\end{proof}

\begin{Lemma}\label{lem:odd-w-repeated-values}
If an irreducible factor $f(w,\lambda)$ of $\chi_n$ is not even in $w$, then $f(w,\lambda)$ and $f(-w,\lambda)$ are distinct factors with the same specialization at $w=0$.
Hence $f(0,\lambda)^2$ divides \[  \prod_i(\lambda+a_i).\]
In particular, the diagonal multiset must contain enough repeated values to support the degree of $f(0,\lambda)$, counted with multiplicity.
\end{Lemma}
\begin{proof}
The involution $w\mapsto -w$ sends factors to factors.
If $f$ is not even, then $f(w,\lambda)$ and $f(-w,\lambda)$ are distinct, but they have the same specialization at $w=0$.
Therefore, their product contributes a square factor to $\chi_n(0,\lambda)$.
\end{proof}

On the pairwise-distinct diagonal stratum, Lemma~\ref{lem:hensel-subsets} allows us to attach a unique subset $S\subset[n]$ to any monic factor by specializing $w=0$.
Following~\cite{shapiro2026reducibilityspectralcurvesfinite}, we call this subset the Hensel subset of the factor.
Equivalently, a factorization has Hensel subset $S$ if one of its monic factors specializes to $\prod_{i\in S}A_i$ at $w=0$, or to the same product after setting $t=0$ for $P_n$.

\begin{Lemma}\label{lem:orbit}
Let $P(t,\lambda)\in \CC[t,\lambda]$ be monic in $\lambda$ and irreducible in $\CC[t,\lambda]$, and set
\[  X(w,\lambda)=P(w^2,\lambda).\]
If $X$ is reducible, then its irreducible factors are permuted by $w\mapsto -w$.
The product over any orbit descends to $\CC[t,\lambda]$, and thus is a factor of $P$.
Therefore, if $P$ is irreducible, then either $X$ is irreducible or
\[  X(w,\lambda)=f(w,\lambda)f(-w,\lambda).\]
\end{Lemma}

\begin{proof}
The involution $w\mapsto -w$ acts on the irreducible factors of $X$.
The product over an orbit is invariant under this involution, hence belongs to $\CC[w^2,\lambda]=\CC[t,\lambda]$.
Since $P$ is irreducible, there can be only one orbit unless $X$ splits as a conjugate pair.
\end{proof}

\begin{Corollary}\label{cor:chi-from-P}
If $P_n$ is irreducible and $n$ is odd, then $\chi_n$ is irreducible.
If $P_n$ is irreducible and $n$ is even, then either $\chi_n$ is irreducible or
\[  \chi_n(w,\lambda)=f(w,\lambda)f(-w,\lambda),\]
where both factors have $\lambda$-degree $n/2$.

Moreover, if $g(t,\lambda)$ is an irreducible factor of $P_n$ of odd $\lambda$-degree, then $g(w^2,\lambda)$ is irreducible in $\CC[w,\lambda]$.
\end{Corollary}
\begin{proof}
Apply Lemma~\ref{lem:orbit}.
In odd $\lambda$-degree, a conjugate-pair factorization would have to split the degree evenly, which is impossible.
\end{proof}

\section{Newton Polytope Constraints on Factors}\label{sec:newton-polytope}

In this section Newton polytopes are taken in the $(t,\lambda)$-exponent coordinates.
After a generic shift in $\lambda$, the Newton polytope is nearly triangular and its Minkowski summands strongly restrict factor degrees.

\begin{Theorem}\label{thm:newton-polytope-shape}
Suppose that no $c_i=b_i^2$ is zero.
Then we have the following.
\begin{enumerate}
    \item If $n=2m$, then for generic $s\in\CC$,
    \[ \calN(P_n(t,\lambda+s))=\conv\{(0,0),(m,0),(0,2m)\}.  \]
    \item If $n=2m+1$, then for generic $s\in\CC$,
    \[ \calN(P_n(t,\lambda+s))=\conv\{(0,0),(m,0),(m,1),(0,2m+1)\}, \]
    as long as
    \[ \sum_{r=0}^m \left(\prod_{q=1}^r c_{2q-1}\right)\left(\prod_{q=r+1}^m c_{2q}\right)\neq0. \]
\end{enumerate}

Note that $P_n(t,\lambda)$ factors if and only if $P_n(t,\lambda+s)$ factors.
\end{Theorem}
\begin{proof}
For $n=2m$, the point $(0,2m)$ comes from $\lambda^{2m}$, and $(m,0)$ comes from the permutation  
\[  (1,2),(3,4),\ldots,(2m-1,2m),\]
in the Leibniz expansion of the determinant whose coefficient is \[(-1)^m c_1c_3\cdots c_{2m-1}\neq0.\]
A generic shift $\lambda\mapsto\lambda+s$ makes the constant term nonzero, giving the claimed triangle.
Since a triangle is only homothetically decomposable, any Newton polytope of a factor is homothetic to this triangle, up to translation and the degenerate point case.

For $n=2m+1$, the point $(0,2m+1)$ comes from $\lambda^{2m+1}$.
The maximum matchings have size $m$ and leave one vertex unmatched.
The coefficient of $t^m\lambda$ is
\[(-1)^m \sum_{r=0}^m\left(\prod_{q=1}^r c_{2q-1}\right)\left(\prod_{q=r+1}^m c_{2q}\right).\]
By hypothesis, this is nonzero, so $(m,1)$ occurs.
A generic shift in $\lambda$ then makes the coefficient of $t^m$ nonzero, giving $(m,0)$, and also makes the constant term nonzero.
\end{proof}

\begin{Corollary}\label{cor:even-factors}
Under the assumptions of Theorem~\ref{thm:newton-polytope-shape}:
When $n$ is even, all factors of $P_n(t,\lambda)$ have even degree in $\lambda$.
If, moreover, the diagonal entries are pairwise distinct, then every quadratic factor of $P_n(t,\lambda)$ has the form
\[ (\lambda+a_i)(\lambda+a_j)-\eta t.\]
\end{Corollary}
\begin{proof}
For $n=2m$, the Newton polytope is the triangle
\[  \conv\{(0,0),(m,0),(0,2m)\}.\]
Every Minkowski summand of a triangle is homothetic to that triangle, up to translation and the degenerate point case.
If a factor has $\lambda$-degree $d$, then its Newton polytope has vertical height $d$.
The homothety ratio is therefore $d/(2m)$, so its horizontal length is $d/2$.
Since this is a lattice polytope, $d/2 \in \ZZ$, and hence $d$ is even.

Assume now that the diagonal entries are pairwise distinct.
For $d=2$, Lemma~\ref{lem:universal-support} gives
\[ F(t,\lambda)=F_0(\lambda)+\eta' t.\]
Then $F_0$ is a product $A_iA_j$.
Renaming $-\eta'$ as $\eta$ gives the displayed form.
\end{proof}

\begin{Lemma}\label{lem:odd-minkowski-summands}
Assume the hypotheses of Theorem~\ref{thm:newton-polytope-shape}.
Let $n=2m+1$, and assume
\[\sum_{r=0}^m\left(\prod_{q=1}^r c_{2q-1}\right)\left(\prod_{q=r+1}^m c_{2q}\right)\neq 0.\]
Then, after the same generic $\lambda$-shift used in Theorem~\ref{thm:newton-polytope-shape}, every nonconstant factor of $ P_n(t,\lambda)$ has a Newton polytope, up to translation, of one of the following two forms:
\[ \text{conv}\{(0,0),(r,0),(0,2r)\},\]
or
\[ \text{conv}\{(0,0),(0,2r+1),(r,1),(r,0)\}.\]
In particular, every factor has either even $\lambda$-degree $2r$ or odd $\lambda$-degree $2r+1$, and at most one factor in a complete factorization can have odd $\lambda$-degree.
\end{Lemma}

\begin{proof}
By Theorem~\ref{thm:newton-polytope-shape} and the displayed nonvanishing condition, the Newton polytope of $ P_{2m+1}$, after the chosen generic shift in $\lambda$, is
\[  Q_m  =  \text{conv}\{(0,0),(m,0),(m,1),(0,2m+1)\}.\]
Indeed, the sum shown is the coefficient of the top $t^m\lambda$-term, so the vertex $(m,1)$ is present; the generic shift ensures that the corresponding $t^m\lambda^0$-coefficient is also nonzero, giving the vertex $(m,0)$.

Let
\[ P_{2m+1}=FG\]
with $F,G$ monic in $\lambda$.
Then
\[ N( P_{2m+1})=N(F)+N(G),\]
where $N(\cdot)$ denotes the Newton polytope.
Hence $N(F)$ is a Minkowski summand of $Q_m$.

Every edge direction of a Minkowski summand of a polytope is parallel to an edge direction of the polytope.
Therefore, the edge directions of $N(F)$ are among the four primitive directions
\[ (1,0),\qquad (0,1),\qquad (-1,2),\qquad (0,-1),\]
which are the edge directions of $Q_m$.

Let the lattice lengths of the edges of $N(F)$ in these directions be
\[  r,\quad \varepsilon,\quad s,\quad d,\]
with length $0$ allowed if the corresponding edge is absent.
The sum of the edge vectors around the polytope must be zero, so
\[  r(1,0)+\varepsilon(0,1)+s(-1,2)+d(0,-1)=0.\]
Comparing coordinates gives
\[ r=s,\qquad d=2r+\varepsilon.\]

The corresponding edge of $Q_m$ in the direction $(0,1)$ has lattice length $1$.
Since edge lengths add under Minkowski sum and are nonnegative integers for lattice polytopes, we must have
\[ \varepsilon\in\{0,1\}.\]

If $\varepsilon=0$, then the right vertical edge is absent.
Up to translation, the polytope has edge lengths
\[  r,\ 0,\ r,\ 2r,\]
and hence is
\[ \text{conv}\{(0,0),(r,0),(0,2r)\}.\]

If $\varepsilon=1$, then the polytope has edge lengths
\[ r,\ 1,\ r,\ 2r+1,\]
and hence, up to translation, is
\[  \text{conv}\{(0,0),(r,0),(r,1),(0,2r+1)\}.\]

Since the $(0,1)$-edge length of $Q_m$ is $1$, the values of $\varepsilon$ over all factors in a complete factorization sum to $1$.
Thus, we see that at most one factor can have $\varepsilon=1$, i.e. at most one factor can have odd $\lambda$-degree.
\end{proof}

\begin{Corollary}\label{Cor:oddPolytopeConsequences}
Assume $n=2m+1$, no $c_i$ is zero, and
\[\sum_{r=0}^m\left(\prod_{q=1}^r c_{2q-1}\right)\left(\prod_{q=r+1}^m c_{2q}\right)\neq 0.\]
Any factorization of $P_n(t,\lambda)$ in odd length has at most one factor of odd degree in $\lambda$.

Consequently, in the irreducible factorization of $\chi_n(w,\lambda)$ in $\CC[w,\lambda]$, all odd $\lambda$-degree factors occur in $w\mapsto -w$ conjugate pairs, except for the factor descending from the unique odd-degree factor of $P_n$.

Moreover, any even-degree factor of $P_n(t,\lambda)$ has Newton polytope of the form
\[ \conv\{(0,0),(r,0),(0,2r)\}.\]
In particular, as in the even case, all quadratic factors of $P_n(t,\lambda)$ have the form $A_pA_q-\alpha t$.
\end{Corollary}

\section{On Constant Branches}\label{sec:constant-branches}

\begin{Lemma}\label{lem:conFracRep}
We may write:       \[\frac{P_{1,j-1}}{P_{1,j}}= \cfrac{1}{L_1-\cfrac{t c_{j-1}}{L_2-\cfrac{t c_{j-2}}{\ddots-\cfrac{t c_1}{L_j}}}}\]

Where $L_i=\lambda+a_{j-i+1}$.
\end{Lemma}
\begin{proof}
This is obtained by repeatedly dividing the recurrence
\[ P_{1,k}=(\lambda+a_k)P_{1,k-1}-tc_{k-1}P_{1,k-2} \]
by $P_{1,k-1}$, starting from $k=j$ and decrementing.
\end{proof}

\begin{Lemma}\label{lem:contFracEquivlance}
Let $M(t)$ and $N(t)$ be two rational functions of the form 
\[ M = \cfrac{1}{L_1-\cfrac{t u_2}{L_2-\cfrac{t u_3}{\ddots-\cfrac{t u_l}{L_l}}}}, \qquad N = \cfrac{1}{R_1-\cfrac{t v_2}{R_2-\cfrac{t v_3}{\ddots-\cfrac{t v_r}{R_r}}}},\]
where $L_i, R_i, u_i, v_i$ are all nonzero.
Suppose that $\alpha M + \beta N = 0$ as a rational function, where $\alpha, \beta \neq 0$, then $l = r$.
Moreover, for this equality to be satisfied, we must have that $\beta = - \alpha \frac{R_1}{L_1}$ and $v_{s} = u_{s}\frac{R_{s-1}R_{s}}{L_{s-1} L_{s}}$ for $s = 2,\dots, l$.
\end{Lemma}
\begin{proof}
Define \[M_i(t) = \cfrac{1}{L_i-\cfrac{t u_{i+1}}{L_{i+1}-\cfrac{t u_{i+2}}{\ddots-\cfrac{t u_l}{L_l}}}}, \qquad N_j(t) = \cfrac{1}{R_j-\cfrac{t v_{j+1}}{R_{j+1}-\cfrac{t v_{j+2}}{\ddots-\cfrac{t v_r}{R_r}}}}.\]
    
First, notice that letting $t = 0$ gives $\beta = - \alpha \frac{R_1}{L_1}$.

Substituting this back in gives us $\alpha M-\alpha\frac{R_1}{L_1}N=0$, and so we have that $M=\frac{R_1}{L_1}N$. 
Inverting both sides gives us $\frac{1}{M}=\frac{L_1}{R_1}\frac{1}{N}$, where $\frac{1}{M}=L_1-tu_2M_2(t)$ and $\frac{1}{N}=R_1-tv_2N_2(t)$.

Thus $t(u_2M_2-\frac{L_1}{R_1}v_2N_2)=0$, so factoring out $t$ forces $u_2M_2-\frac{L_1}{R_1}v_2N_2=0$.
Note that this is the same type of relation that we started with.
Setting $t=0$ again gives us that $v_2=u_2\frac{R_1R_2}{L_1L_2}$,

Plugging this in for $v_2$ and repeating the same process, we obtain $u_3M_3 - \frac{L_2}{R_2}v_3N_3=0$, or more generally, we see that $u_s M_s - \frac{L_{s-1}}{R_{s-1}}v_sN_s = 0$, and so $v_s = u_{s}\frac{R_{s-1}R_{s}}{L_{s-1} L_{s}}$.

After interchanging $M$ and $N$, it is enough to rule out $l<r$.
In that case $M_l=\frac{1}{L_l}$ and $N_l=\frac{1}{R_l-tv_{l+1}N_{l+1}}$ are linearly dependent.
However, this would mean there is a nonconstant function that is given by a constant function.
\end{proof}

\begin{Theorem}\label{thm:constBranchClass}
Suppose that the $a_i$ are pairwise distinct and $c_i\neq0$.
Let $n=2m+1 \geq 3$.
If $P_n(t,\lambda)$ has a factor of the form $\lambda+\alpha$, then $\alpha=a_{m+1}$.

Moreover, $\lambda+a_{m+1}$ divides $P_n(t,\lambda)$ if and only if the $c_i$ satisfy the following relations:
\[ c_{m+1}=-c_m\frac{a_{m+2}-a_{m+1}}{a_m-a_{m+1}}, \]
and
\[c_{m+s}=c_{m+1-s} \frac{(a_{m+s}-a_{m+1})(a_{m+s+1}-a_{m+1})}{(a_{m+2-s}-a_{m+1})(a_{m+1-s}-a_{m+1})}, \qquad s=2,\ldots,m.  \]
\end{Theorem}
\begin{proof}
Suppose that $\lambda+a_i$ is a factor of $P_n(t,\lambda)$.
Expanding along the $i$-th row/column gives
\[ P_n(t,-a_i)= -t c_{i-1} P_{1,i-2}(t,-a_i) P_{i+1,n}(t,-a_i) - t c_i P_{1,i-1}(t,-a_i) P_{i+2,n}(t,-a_i). \]
Thus, if $\lambda+a_i$ is a constant branch, then
\[ c_{i-1} P_{1,i-2}(t,-a_i) P_{i+1,n}(t,-a_i)  + c_i P_{1,i-1}(t,-a_i) P_{i+2,n}(t,-a_i)=0. \]
    
If $i=1$ or $n$, this is immediately impossible: sending $t$ to $0$ would force a repeated diagonal entry.

Let
\[M=\frac{P_{1,i-2}}{P_{1,i-1}},\qquad N=\frac{P_{i+2,n}}{P_{i+1,n}},\]
evaluated at $\lambda=-a_i$.
The constant branch condition becomes
\[c_{i-1}M+c_iN=0 \]

The two rational functions $M(t)$ and $N(t)$ are continued fractions of depths $i-1$ and $n-i$.
By Lemma~\ref{lem:contFracEquivlance}, we must have that these depths agree.
Thus $i-1=n-i$, and since $n=2m+1$, this forces $i=m+1$.
Applying Lemma~\ref{lem:contFracEquivlance} gives the displayed relations among the $c_i$.

Conversely, when $i=m+1$, Lemma~\ref{lem:contFracEquivlance} shows that the displayed relations are exactly the conditions under which $c_mM+c_{m+1}N=0$.
Hence, they are also sufficient for $\lambda+a_{m+1}$ to divide $P_n(t,\lambda)$.
\end{proof}
It is clear from the proof that we may make the following relaxation of the above theorem.
\begin{Corollary}\label{cor:middle-branch-unique}
Suppose that $n = 2m+1$. If $\lambda+\alpha$ is a linear factor of $P_n(t,\lambda)$ and $a_i=\alpha$ for a unique index $i\in[n]$, then $i=m+1$.
\end{Corollary}
Moreover, this gives us the following regarding the dimension of constant branches in the locus where no $c_i=0$ and the $a_i$ are pairwise distinct.

\begin{Corollary}\label{cor:constant-branch-codim-distinct}
When $n=2m+1$, for a fixed pairwise distinct diagonal $a$, the locus of couplings admitting constant branches has codimension $m$.
\end{Corollary}
\begin{proof}
In this space the couplings $c_{m+1},\ldots,c_{2m}$ are uniquely determined by $c_1,\ldots,c_m$.
\end{proof}

When multiple values of $a_i$ are allowed to take the same value, constant branches are much easier to come by.
Although we will not attempt to give a complete classification of all such mechanisms, we illustrate one additional mechanism.
\begin{Theorem}\label{Thm:constantsSpec}
Let $n=2m+1$.
Suppose that $a_1=a_3=\cdots=a_{n-2}=a_n$.
Then $\lambda+a_n\mid P_n(t,\lambda)$.
\end{Theorem}
\begin{proof}
Follows immediately from the Leibniz expansion.
\end{proof}

We now show that for every fixed connected coupling, the codimension of parameters that admit constant branches in the diagonal variables is $\lfloor n/2\rfloor$. 

\begin{Theorem}\label{thm:constant-branches-fixed-couplings}
Let $n=2m+1$, and fix a choice of $c$ such that $c_i\neq0$ for all $i$.
The codimension of the space of potentials that admit constant branches is exactly $m$.
\end{Theorem}
\begin{proof}
Suppose that $\lambda+a_j$ is a constant branch.
Then
\[  P_n(t,-a_j)\equiv0, \]
which gives $m$ coefficient equations in the variable $t$.
To prove that the codimension is exactly $m$, we compute the dimension of the incidence where a constant branch is allowed to occur at an unspecified value.

Let $\lambda=-\alpha$, $z_i=a_i-\alpha$, and
\[ K_n(t,z)=P_n(t,-\alpha). \]
Then $K_0=1$, $K_1=z_1$, and
\[  K_k=z_kK_{k-1}-tc_{k-1}K_{k-2}.\]
We write $K_k(t,z,c)$ to allow us to modify the off-diagonal terms.

Define
\[X_n(c)=\{z\in\CC^n:K_n(t,z)\equiv0\}.\]
We first show that
\[  \dim X_{2m+1}(c)=m.\]

As the constant term of $K_n$ as a polynomial in $t$ is $\prod z_i$, necessarily some $z_i$ must vanish.
We analyze each piece $X_n(c)\cap\{z_i=0\}$.

If $i=1$ or $i=n$, then
\[  K_n(t,0,z_2,\ldots,z_n)=-tc_1K_{n-2}(t,z_3,\ldots,z_n)\]
or
\[ K_n(t,z_1,\ldots,z_{n-1},0)=-tc_{n-1}K_{n-2}(t,z_1,\ldots,z_{n-2}), \]
respectively.
Thus, the endpoint piece is a product of a free coordinate and the corresponding $X_{n-2}(c')$.

If $2\le i\le n-1$, then setting $z_i=0$ contracts the two neighboring vertices into the linear combination
\[  c_i z_{i-1}+c_{i-1}z_{i+1}.\]
More precisely,
\begin{equation*}
\begin{aligned}
&K_n(t,z_1,\ldots,z_{i-1},0,z_{i+1},\ldots,z_n)\\
&\qquad =-t\,K_{n-2}\bigl(t,z_1,\ldots,z_{i-2},
c_i z_{i-1}+c_{i-1}z_{i+1},z_{i+2},\ldots,z_n;c'\bigr),
\end{aligned}
\end{equation*}
for the induced contracted coupling vector $c'$.

Note that if $i=2$, then the left coupling in the contracted chain does not appear, and if $i=n-1$, then the right coupling does not appear.

This identity follows by a straightforward computation that we include for the reader's convenience:
Recall that
\[K_{1,n}|_{z_i=0} = -t(c_{i-1}K_{1,i-2}K_{i+1,n}+ c_iK_{1,i-1}K_{i+2,n}).\]

Expand
\[ K_{i+1,n}=z_{i+1}K_{i+2,n}-tc_{i+1}K_{i+3,n},\qquad  K_{1,i-1}=z_{i-1}K_{1,i-2}-tc_{i-2}K_{1,i-3}. \]

Substituting gives the contracted continuant displayed above, which is exactly the determinant of the Jacobi matrix described.
Thus, each entry of $c'$ is either an original coupling or a product of two original couplings. In particular, $c'$ is connected.

Equivalently, on the hyperplane $\{z_i=0\}$, there is a linear map
\[  \pi_i:\{z_i=0\}\longrightarrow \CC^{n-2} \]
whose coordinates are those of the contracted chain.
This map is surjective and has one-dimensional fibers: for $2\le i\le n-1$, the only nontrivial coordinate is $c_i z_{i-1}+c_{i-1}z_{i+1}$, and both $c_{i-1}$ and $c_i$ are nonzero.

Now we can prove by induction that $X_{2m+1}(c)$ has dimension $m$.
The base case $n=1$ is immediate.

Now assume that $n=2m+1$.
Since the constant term of $K_n$ is $\prod z_i$, the condition $K_n(t,z)\equiv0$ implies that some $z_i=0$.
By the work above, for each such $i$ there is a contracted coupling vector $c'$ and a surjective contraction map $\pi_i$, with one-dimensional fibers, such that
\[  X_n(c)\cap\{z_i=0\}=\pi_i^{-1}(X_{n-2}(c')). \]
By induction, $\dim X_{n-2}(c')=m-1$.
Hence, each piece has dimension $m$, and so $\dim X_n(c)=m$.
Thus, $X_n(c)$ has codimension $m+1$ in $\CC^n$.

Finally, consider
\[  Y=\{(\alpha,z):z\in X_{2m+1}(c)\}\subset \CC\times\CC^{2m+1}.\]
It has dimension $m+1$.
The map
\[  Y\longrightarrow \CC^{2m+1},\qquad  (\alpha,z)\longmapsto a=z+\alpha(1,\ldots,1) \]
has image exactly the potential vectors admitting a constant branch.
Its fibers are finite, because $z\in X_{2m+1}(c)$ forces some $z_i=0$, so for a fixed image $a$ the possible values of $\alpha$ belong to the finite set $\{a_1,\ldots,a_{2m+1}\}$.
Therefore, the image has dimension $m+1$, and its codimension in $\CC^{2m+1}$ is $m$.
\end{proof}

\section{Strata of Pairwise Distinct Diagonal Entries}\label{sec:pairwise-distinct-strata}
Fix pairwise-distinct diagonal entries $a_1,\ldots,a_n$.
We interpret $[n]$ as the vertex set of the path graph with edge $i$ given by $(i,i+1)$.
For a subset $S\subset [n]$, define its boundary edge set by
\[ \partial S := \{\, i\in\{1,\ldots,n-1\} : |\{i,i+1\}\cap S|=1 \,\}.\]
Equivalently, $i\in\partial S$ if and only if exactly one endpoint of the edge $(i,i+1)$ lies in $S$.
For a nonempty proper subset $S\subset[n]$, define
\[ \ell_S(c) := \sum_{k=1}^{n-1} \frac{\mathbf 1_{k+1\in S}-\mathbf 1_{k\in S}}{a_{k+1}-a_k}\,c_k. \]

\begin{Theorem}\label{thm:stratcontainment}
Assume $a_i\neq a_j$ for any $i\neq j$.
If $P_n(t,\lambda)$ has a factor whose Hensel subset at $t=0$ is $S$, then $\ell_S(c)=0$.
\end{Theorem}
\begin{proof}
Write $A_i=\lambda+a_i$ and $P^{[r]}_n=[t^r]P_n$, so that
\[ P^{[0]}_n=\prod_{i=1}^n A_i.\]
By the Leibniz expansion,
\begin{equation}\label{eq:Pwcoeff}
  P^{[1]}_n=-\sum_{i=1}^{n-1}c_i\prod_{j\neq i,i+1}A_j.
\end{equation}
Let \[F_0=\prod_{p\in S}A_p, \qquad G_0=\prod_{q\notin S}A_q,\] and suppose
\[ P_n=(F_0+tF_1+\cdots)(G_0+tG_1+\cdots). \]
Notice that
\begin{equation}\label{eq:wfactor}
  F_1G_0+F_0G_1=P^{[1]}_n.
\end{equation}
By Lemma~\ref{lem:universal-support}, we have $\deg F_1\le |S|-2$.
Evaluating \eqref{eq:wfactor} at $\lambda=-a_p$, $p\in S$, gives
\[  F_1(-a_p)=\frac{P^{[1]}_n(-a_p)}{G_0(-a_p)}. \]

The unique polynomial of degree at most $|S|-1$ interpolating $F_1$ through the values $\{ -a_i \mid i \in S \}$ is
\[  I(\lambda)=\sum_{p \in S} F_1(-a_p) L_p(\lambda), \]
where
\[  L_p(\lambda)  =  \prod_{q \in S \smallsetminus \{p\}}  \frac{\lambda + a_q}{a_q - a_p}  =  \frac{F_0(\lambda)}{(\lambda+ a_p)F'_0(-a_p)}.\]

Thus,
\[  I(\lambda)  =  \sum_{p \in S} F_1(-a_p)  \frac{F_0(\lambda)}{(\lambda+ a_p)F'_0(-a_p)}  =  \sum_{p \in S}  \frac{P^{[1]}_n(-a_p)}{G_0(-a_p)}  \frac{F_0(\lambda)}{(\lambda+ a_p)F'_0(-a_p)}.\]
Since $\deg F_1 \leq  |S|-2$, uniqueness of interpolation gives $I=F_1$. 
Consequently, the coefficient of $\lambda^{|S|-1}$ in $I$ vanishes.

As $\frac{F_0(\lambda)}{\lambda+a_p}$ is monic, $I(\lambda)$ has the coefficient of $\lambda^{|S|-1}$ given by
\[\sum_{p\in S}  \frac{P^{[1]}_n(-a_p)}{F'_0(-a_p)G_0(-a_p)}  =  \sum_{p\in S}  \frac{P^{[1]}_n(-a_p)}{(P^{[0]}_n)'(-a_p)}  = \sum_{p\in S}  \text{Res}_{\lambda = -a_p}  \frac{P^{[1]}_n(\lambda)}{P^{[0]}_n(\lambda)}.\]
Here $\text{Res}$ denotes the standard residue.

Notice that
\begin{equation*}
\begin{aligned}
\frac{ P_n^{[1]}}{ P_n^{[0]}}  &=  -\sum_{i=1}^{n-1} \frac{c_i}{A_iA_{i+1}},\\ 
  \text{Res}_{\lambda = -a_i} \frac{c_i}{A_iA_{i+1}}  &= \frac{c_i}{a_{i+1}-a_i},\\
  \text{Res}_{\lambda = -a_{i+1}} \frac{c_i}{A_iA_{i+1}}  &=  \frac{c_i}{a_i-a_{i+1}}.
\end{aligned}
\end{equation*}
In particular, if both $i$ and $i+1$ are in $S$, or if neither are in $S$, then
\[ \sum_{p\in S} \text{Res}_{\lambda = -a_p} \left(-\frac{c_i}{A_iA_{i+1}}\right)=0.\]
Otherwise, if $i \in S$ and $i+1 \not\in S$, we have
\[ \sum_{p\in S}\text{Res}_{\lambda = -a_p}\left(-\frac{c_i}{A_iA_{i+1}}\right) = - \frac{c_i}{a_{i+1}-a_i},\]
whereas if $i\notin S$ and $i+1\in S$, we have
\[ \sum_{p\in S}  \text{Res}_{\lambda = -a_p} \left(-\frac{c_i}{A_iA_{i+1}}\right) = -\frac{c_i}{a_i-a_{i+1}}.\]
Thus, this vanishing coefficient is exactly $\ell_S(c)=0$.
\end{proof}

\begin{Corollary}\label{cor:first-order-obstruction-hyperplanes}
If $\partial S$ has only one element, then $\ell_S(c)=0$ forces some $c_i=0$. For such $S$, there are no corresponding factorizations in $T_c$.
In particular, any proper Hensel set $S$ given by an initial or terminal interval (i.e. 
$\{1,\dots, r\}$ or $\{r,\dots, n\}$) can only occur if there is a cut.

Moreover, $\ell_{S^c}(c)=-\ell_S(c)$.
\end{Corollary}

Next we show that the collection of $P_n$ that have a factor whose Hensel subset is $S$ at $t=0$ must be a proper subset of $\{\ell_S(c) = 0\}$.

\begin{Lemma}\label{lem:monicreduce}
Fix $a\in\CC^n$ and a subset $S\subset[n]$. Set
\[ F_{0,S}=\prod_{i\in S}(\lambda+a_i), \qquad G_{0,S}=\prod_{i\notin S}(\lambda+a_i).\]
The set of coupling vectors $c\in\CC_c^{n-1}$ for which there exist
monic polynomials $F,G$ satisfying
\[ P_n(t,\lambda)=F(t,\lambda)G(t,\lambda), \qquad F(0,\lambda)=F_{0,S},  \qquad  G(0,\lambda)=G_{0,S},\]
is Zariski closed.
\end{Lemma}
\begin{proof}
As factors are necessarily monic in $\lambda$, this follows from the standard projective factor-incidence construction, in the same spirit as the closedness argument used in~\cite{faust2026genericirreducibilityblochvarieties}.
In particular, let $d=|S|$, and suppose that $F$ and $G$ are factors of $P_n(t,\lambda)$, where $F$ has prescribed specialization $F_{0,S}$.
That is,
    \[ F_0=\prod_{i\in S}(\lambda+a_i),\qquad G_0=\prod_{i\notin S}(\lambda+a_i).  \]
By Lemma~\ref{lem:universal-support}, any factor of degree $d$ has support contained in \[ \Sigma_d=\{(r,j):j+2r\le d\},\] and the complementary factor has support contained in $\Sigma_{n-d}$.

Let $V_d$ and $V_{n-d}$ be the corresponding vector spaces of polynomials in $t,\lambda$, and let $\ell_F,\ell_G$ denote the coefficients of $\lambda^d$ and $\lambda^{n-d}$.
Consider the closed incidence variety
\[ \calX_S\subset \CC_c^{n-1}\times\PP(V_d)\times\PP(V_{n-d})\]
defined by
\[ fg=\ell_F(f)\ell_G(g) P_n(t,\lambda;c),\]
\[ [t^0]f=\ell_F(f)F_0,\qquad [t^0]g=\ell_G(g)G_0.\]
As the product $\PP(V_d) \times \PP(V_{n-d})$ is proper, so is the natural projection $ \CC_c^{n-1}\times\PP(V_d)\times\PP(V_{n-d}) \to \CC_c^{n-1}$. Its restriction to $\calX_S$ is also proper, and thus has closed image. This image is the collection of $c$ such that $P_n$ admits a factor $F$ with specialization $F_{0,S}$.

Indeed, a monic factorization gives a point of $\calX_S$.
Conversely, if $(c,[f],[g])\in\calX_S$, then $fg\neq0$, so $\ell_F(f)\ell_G(g)\neq0$.
Normalizing by these two nonzero leading coefficients gives monic factors $F,G$ with \[ FG= P_n,\qquad F^{[0]}=F_0,\qquad G^{[0]}=G_0.\] 
Thus, the prescribed-specialization locus is closed.
\end{proof}
\begin{Remark}
When the $a_i$ are pairwise distinct, the preceding locus is exactly the fixed-Hensel-subset locus associated with $S$. 
When diagonal values collide, the subset $S$ need not be unique; different subsets may determine the same specialization polynomial $F_{0,S}$.
\end{Remark}

\begin{Lemma}\label{lem:nonadjedges}
Assume that the $a_i$ are pairwise distinct.
Let $S$ be a nonempty proper subset such that neither $S$ nor $S^c$ has size $1$ and such that $\partial S$ has at least two elements.
Let $\mathcal{R}_S\subset \CC_c^{n-1}$ be the affine closed locus where $P_n(t,\lambda)$ has a factor with Hensel subset $S$, and set $H_S=\{\ell_S=0\}$.
Then $\mathcal{R}_S\neq H_S$.
\end{Lemma}
\begin{proof}
First note that $\partial S$ has two non-adjacent $i$.
Indeed, otherwise $\partial S$ would only have two $i$ which are adjacent, and then $S$ or $S^c$ must be a single element.

Now choose $p$, $q$ to be two non-adjacent elements of $\partial S$.
The coefficients of $c_p$ and $c_q$ are then nonzero in $\ell_S$.
    
Let us now fix nonzero values for $c_p$ and $c_q$, with all other $c_i=0$, such that $\ell_S(c)=0$.
At this cut point, $P_n(t,\lambda)$ becomes a product of
\[ A_pA_{p+1}-tc_p,\qquad A_qA_{q+1}-tc_q \]
and constant linear factors.
Since $c_p,c_q \neq 0$, these two quadratics are irreducible.
Because $S$ contains exactly one endpoint of each of the two chosen boundary edges, no factor at this cut point can have Hensel subset $S$: such a factor would split both irreducible quadratics.
Thus, this point lies in $H_S\smallsetminus \calR_S$.
Therefore $\mathcal{R}_S\neq H_S$.
\end{proof}

\begin{Theorem}\label{thm:no-divisors-fixed-distinct}
Fix $n\ge4$, and pairwise distinct diagonal entries.
Then
\[ \text{codim}_{T_c}\{c\in T_c:P_n(t,\lambda)\text{ is reducible}\}\ge2. \]
Equivalently, over the pairwise-distinct diagonal locus, the only codimension $1$ components of reducible parameters are the hyperplanes $c_i = 0$.
\end{Theorem}
\begin{proof}
By Theorem~\ref{thm:stratcontainment}, for each proper Hensel subset $S$, the corresponding collection of $c$ such that $P_n$ has a factor with Hensel subset $S$ is contained in $\{\ell_S=0\}$.
If $S$ has only one boundary edge, this forces some $c_i = 0$ and so this stratum misses $T_c$.

If $S$ or $S^c$ has size one, the factorization has a constant branch factor.
In even length this is impossible on the connected pairwise-distinct stratum by Corollary~\ref{cor:even-factors}.
In odd length, the collection of parameters with constant branch factors has codimension $\lfloor n/2\rfloor$ on this stratum by Corollary~\ref{cor:constant-branch-codim-distinct}.
For $n\ge4$, this is at least codimension $2$ in $T_c$.

Finally, all remaining $S$ satisfy the hypotheses of Lemma~\ref{lem:nonadjedges}.
Let $\calR_S\subset \CC_c^{n-1}$ be the affine closed fixed-$S$ locus from Lemma~\ref{lem:monicreduce}, and set $H_S=\{\ell_S=0\}$.
By Theorem~\ref{thm:stratcontainment},
\[ \calR_S\subseteq H_S.\]
The polynomial $\ell_S$ is a nonzero linear form in the coupling variables, hence $H_S$ is an irreducible affine hyperplane.
Since $\ell_S$ has at least two nonzero coordinate terms, $H_S$ is not contained in any coordinate hyperplane, and therefore $H_S\cap T_c$ is dense in $H_S$.
Lemma~\ref{lem:nonadjedges} gives $\calR_S\neq H_S$.
Therefore $\calR_S$ is a proper closed subset of the irreducible hyperplane $H_S$, and so it has codimension at least $2$ in $\CC_c^{n-1}$.
Intersecting with the open torus $T_c$ preserves this codimension bound.

As the reducibility locus is a finite union of all possible Hensel sets which are at least codimension $2$ in $T_c$, we conclude that the locus itself is at least codimension $2$ in $T_c$.
\end{proof}

\section{Generic Irreducibility for Fixed Diagonal Data}\label{sec:generic-fixed-diagonal}

In this section we prove the following.
\begin{Theorem}\label{thm:sec8main}
\begin{enumerate}
    \item Let $n=2m+1\ge3$ and $a$ be fixed. For generic $c \in T_c$, $P_n$ is reducible if and only if $a_1 = a_3 = \dots = a_{2m+1}$.
    \item Let $n=2m\ge4$ and $a$ be fixed. For generic $c \in T_c$, $P_n$ is reducible if and only if $a_1 = a_3 = \dots = a_{2m-1}$ and $a_2 = a_4 = \dots = a_{2m}$.
\end{enumerate}
\end{Theorem}
First we prove the right-hand side implies the left-hand side.
\begin{Lemma}\label{lem:parity-constant-reducible}
\begin{enumerate}
    \item Assume $n=2m+1\ge3$. For fixed $a$ and generic $c \in T_c$,
$P_n$ is reducible if
    \[      a_1 = a_3 = \dots = a_{2m+1}.    \]
    \item Assume $n=2m\ge4$. For fixed $a$ and generic $c \in T_c$,
$P_n$ is reducible if
    \[  a_1 = a_3 = \dots = a_{2m-1}   \qquad\text{and}\qquad     a_2 = a_4 = \dots = a_{2m}.   \]
\end{enumerate}
\end{Lemma}
\begin{proof}
Notice (1) follows from Theorem~\ref{Thm:constantsSpec}.

Let us now prove (2).
Suppose that $a_{2r-1} = \alpha$ and $a_{2r} = \beta$ for each $r=1,\dots, m$.
Let $A = \lambda + \alpha$ and $B = \lambda + \beta$.

By reordering the vertices in the Jacobi matrix by listing the odd vertices and then the even vertices, we have
\[ J_n(w)+\lambda I = \begin{pmatrix} A I_m & wD \\ w D^T & B I_m \end{pmatrix},\]
where $D$ is the weighted odd-even incidence matrix.
Taking the Schur complement gives
\[ \chi_n(w,\lambda)=\det(A B I_m - w^2 D^T D),\]
and hence, in the reduced variable $t=w^2$,
\[ P_n(t,\lambda)=\det(A B I_m - t D^T D).\]
If the eigenvalues of $D^TD$ are $\rho_1,\ldots,\rho_m$, then
\[P_{2m}(t,\lambda)=\prod_{i=1}^m(A B - \rho_i t). \]
\end{proof}

To prove the remaining direction of Theorem~\ref{thm:sec8main}, we need the following lemma.

\begin{Lemma}\label{lem:helper8}
Let $n \geq 2$. Assume that $P_n(t,\lambda)$ is irreducible for fixed couplings $(c_1,\dots,c_{n-1})\in \CC^{n-1}$, and assume that
\[P_{n-1}(t,-a_{n+1})\not\equiv 0. \]
Then, $P_{n+1}(t,\lambda)$ is irreducible for generic $c_n\in\CC$.
\end{Lemma}
\begin{proof} 
Suppose that $P_{n+1}$ is reducible for a generic choice of $c_n$, and let $A=\lambda+a_{n+1}$.
As in Lemma~\ref{lem:monicreduce}, for $d = 1, \dots, \lfloor \frac{n+1}{2} \rfloor$ form the projective incidence variety of degree $d$ and degree $n+1-d$ monic factor pairs, using Lemma~\ref{lem:universal-support} to keep the coefficient spaces finite-dimensional.
For each $d$, the image in $\CC$ is closed. As $P_{n+1}$ is generically reducible, this means that the union of these finitely many varieties must be the entire $\CC$. In particular, for some $d$ the image of the corresponding projective incidence variety is all of $\CC$. 

Thus, the degree-$d$ factor locus has a specialization at $c_n=0$.

At $c_n=0$, we have
\[ P_{n+1}=A P_n.\]
By assumption, $P_n$ is irreducible and $A$ is coprime to $P_n$.
Thus, $d$ must be $1$.
Hence $P_{n+1}$ has a linear factor for all $c_n \in \CC$.

By Lemma~\ref{lem:universal-support}, any monic linear factor is independent of $t$.
At $t=0$, this linear factor divides the fixed polynomial
\[ \prod_{i=1}^{n+1}(\lambda+a_i),\]
so its root must be one of the finitely many values $-a_i$.
Note that for each $a_i$, the locus of $c_n \in \CC$ such that $\lambda + a_i$ is a factor of $P_{n+1}$ is algebraic.
Thus, at least one of these linear factors must always be a factor of $P_{n+1}$.
Specializing at $c_n=0$, the only possible linear factor is $A$.
Hence, the generic linear factor is identically $A$, and $A$ would divide $P_{n+1}$ for generic $c_n$.
But
\[ P_{n+1}(t,-a_{n+1})=-tc_nP_{n-1}(t,-a_{n+1}),\]
which is not identically zero by assumption, giving us a contradiction.
\end{proof}
\subsection{Proof of Theorem~\ref{thm:sec8main}.}
\begin{proof}
We prove the claim by induction.

The base cases are clear.
When $n=2$, then
\[ P_2=(\lambda+a_1)(\lambda+a_2)-tc_1,\]
which is irreducible in $\CC[t,\lambda]$ for all $c_1\neq0$.
When $n=3$, then
\[ P_3=(\lambda+a_1)(\lambda+a_2)(\lambda+a_3)  -t(c_1(\lambda+a_3)+c_2(\lambda+a_1)). \]
If $a_1=a_3$, then $\lambda+a_1$ is a factor.
Conversely, a generic reducible cubic must have a linear factor, and by Lemma~\ref{lem:universal-support} this factor is independent of $t$.
Testing the possible roots $-a_1,-a_2,-a_3$ shows that, for generic nonzero $c_1,c_2$, this can only happen when $a_1=a_3$.
    
Now consider $n \geq 4$.
Let $n = 2m$ and suppose that the diagonal is not constant for at least one parity class.
    
If the odd class is nonconstant, then $P_{2m-1}$ is such that its odd entries are nonconstant, and thus it is irreducible for a generic coupling by induction.
Moreover, $P_{2m-2}(t,-a_{2m})\not\equiv0$, since an even-length continuant has a nonzero top matching coefficient independent of $\lambda$.
Thus, all conditions of Lemma~\ref{lem:helper8} are met, and we conclude that $P_n$ is generically irreducible.

If instead the even class is nonconstant, just repeat the same argument on the reverse Jacobi matrix with labeling $(a_n,\dots, a_1)$, $(c_{n-1},\dots, c_1)$.
    
Now suppose that $n=2m+1$ and the odd diagonal class $a_1,a_3,\ldots,a_{2m+1}$ is not constant.
At least one of the two even blocks
\[ P_{1,2m},\qquad P_{2,2m+1} \]
is not parity-constant in the sense of the even case.
Indeed, if both were parity-constant, then the odd entries of $P_{1,2m}$ would make $a_1=a_3=\cdots=a_{2m-1}$, and the even entries of $P_{2,2m+1}$ would make $a_3=a_5=\cdots=a_{2m+1}$, so all odd entries of the original chain would be equal.

First suppose $P_{1,2m}$ is not parity-constant.
By the even case of the induction, $P_{1,2m}$ is irreducible for a generic choice of $c_1,\ldots,c_{2m-1}$.
We claim that
\[ P_{1,2m-1}(t,-a_{2m+1})\not\equiv 0 \]
as a polynomial in $t$ and in the couplings $c_1,\ldots,c_{2m-2}$.
The coefficient of $t^{m-1}$ is a sum over maximum matchings of the path $1,\ldots,2m-1$.
These maximum matchings leave exactly one odd vertex $2r-1$ unmatched.
The corresponding coupling monomials are distinct, and after substituting $\lambda=-a_{2m+1}$, their coefficients are $a_{2r-1}-a_{2m+1}$, up to a common sign.
Hence, the coefficient can vanish identically only if
\[ a_1=a_3=\cdots=a_{2m-1}=a_{2m+1}, \]
contrary to the assumption that the odd class is not constant.
Thus, Lemma~\ref{lem:helper8}, applied to the prefix $P_{1,2m}$ and the terminal vertex $2m+1$, implies that $P_n$ is generically irreducible.

If $P_{1,2m}$ is parity-constant, then $P_{2,2m+1}$ is not parity-constant.
Apply the same argument to the reversed chain
\[ (a_{2m+1},a_{2m},\ldots,a_1), \qquad  (c_{2m},c_{2m-1},\ldots,c_1). \]
This proves generic irreducibility in the remaining odd case.
\end{proof}

\begin{Corollary}~\label{cor:inheritance}
Let \[\chi_n(w,\lambda)=P_n(w^2,\lambda).\]
For $n\ge3$, $P_n$ and $\chi_n$ have the same generic irreducibility
classification for every fixed diagonal vector.

More explicitly, if $n=2m+1$, then $\chi_n$ is generically reducible if and only if \[a_1=a_3=\cdots=a_{2m+1}. \]
If $n=2m \ge 4$, then $\chi_n$ is generically reducible if and only if 
\[ a_1=a_3=\cdots=a_{2m-1}, \qquad a_2=a_4=\cdots=a_{2m}.\]
\end{Corollary}
\begin{proof}
For a fixed diagonal vector, suppose $P_n$ is generically irreducible but $\chi_n(w,\lambda)=P_n(w^2,\lambda)$ is generically reducible.
By Lemma~\ref{lem:orbit}, this can only happen when $n=2m$, and then $\chi_n$ has a balanced $m+m$ factorization
\[ \chi_n(w,\lambda)=F(w,\lambda)F(-w,\lambda).\]
The parity two-color diagonal locus is given by
\[ a_1=a_3=\cdots=a_{2m-1},\qquad a_2=a_4=\cdots=a_{2m}.\]

As $P_n$ is generically irreducible, our fixed diagonal vector must be away from this locus, and thus at least one parity class is nonconstant.

If the odd class is nonconstant, then by the odd case of Theorem~\ref{thm:sec8main} and Corollary~\ref{cor:chi-from-P}, the prefix $\chi_{1,2m-1}$ is generically irreducible.
Fix generic values of the prefix couplings for which $\chi_{1,2m-1}$ is irreducible. By the same argument as given in Lemma~\ref{lem:monicreduce}, but in $\CC[w,\lambda]$, a generic balanced $m+m$ factorization would persist after specializing $c_{2m-1}=0$.
But then
\[ \chi_{1,2m}=(\lambda+a_{2m})\chi_{1,2m-1},\]
whose only proper factor degrees are $1$ and $2m-1$, not $m$.
This is a contradiction.
If the even class is nonconstant, the same argument applies after cutting $c_1=0$ and using the suffix chain.
\end{proof}

\section{Dimension Bound on the Reducibility Locus}\label{sec:dimension-bound}
We are now ready to prove our main result, a lower bound on the codimension of the reducibility locus outside of the cut hyperplanes.

\begin{Theorem}\label{thm:no-divisors-full}
Let \[\calU_n=\CC^n_a\times T_c. \]
For $n\ge4$, the reduced reducible locus
\[ R_n^P= \{(a,c)\in\calU_n:  P_n(t,\lambda)\text{ is reducible in }\CC[t,\lambda]\} \]
has codimension at least $2$.
Moreover, the original reducible locus
\[R_n^\chi= \{(a,c)\in\calU_n: \chi_n(w,\lambda)\text{ is reducible in }\CC[w,\lambda]\} \]
also has codimension at least $2$.
\end{Theorem}
\begin{proof}
Stratify $\CC^n_a$ by equality patterns of the diagonal entries. By the same projective factor-incidence argument as in Lemma~\ref{lem:monicreduce}, both reducible loci are closed. For the continuant,  Theorem~\ref{thm:no-divisors-fixed-distinct} gives us codimension at least $2$ over the pairwise-distinct diagonal stratum.

Now let $D_\pi$ be a diagonal equality stratum.
If $\text{codim}D_\pi\ge2$, then even if the whole family over $D_\pi$ were reducible, it would already have codimension at least $2$ in $\calU_n$.

Thus, only diagonal hypersurfaces of the form $a_p=a_q$ could produce divisors.
By Theorem~\ref{thm:sec8main}, generic reducibility over a fixed diagonal vector requires the relevant parity class to be constant: the odd class in odd length, and both parity classes in even length.
For $n\ge4$, no single equality hyperplane $a_p=a_q$ forces these parity-constant conditions; those conditions have codimension at least $2$ in the diagonal space.
Hence generic connected couplings over any single diagonal equality hyperplane give an irreducible continuant.
Therefore the reducible locus inside each such hyperplane is a proper closed subset, and thus has codimension at least $2$ in $\calU_n$.
Hence
\[\text{codim}_{\calU_n} R_n^P\ge2.\]

On the pairwise distinct stratum, Lemma~\ref{lem:hensel-subsets} identifies the $P_n$ and $\chi_n$ reducible locus. On each equality stratum, Corollary~\ref{cor:inheritance} allows the preceding argument to be repeated for $\chi_n$. 
\end{proof}

\section*{Acknowledgements}
This research was partially supported by NSF grant DMS-2052519.

\section*{Statements and Declarations}
{\bf Conflict of Interest} The author declares no conflicts of interest.

\vspace{0.2in}
{\bf Data Availability} Data sharing is not applicable to this article as no new data were created or analyzed in this study.

\bibliographystyle{alpha} 
\bibliography{main}

\end{document}